\documentclass[12pt,oneside,english]{amsart}
\usepackage{lmodern}
\usepackage[T1]{fontenc}
\usepackage[latin9]{inputenc}
\usepackage{geometry}
\usepackage{babel}
\usepackage{verbatim}
\usepackage{amstext}
\usepackage{amsthm}
\usepackage{amssymb}
\usepackage[numbers]{natbib}
\usepackage[unicode=true,pdfusetitle,
 bookmarks=true,bookmarksnumbered=false,bookmarksopen=false,
 breaklinks=false,pdfborder={0 0 1},backref=false,colorlinks=false]
 {hyperref}

\makeatletter
\numberwithin{equation}{section}
\numberwithin{figure}{section}
\theoremstyle{plain}
\newtheorem{thm}{\protect\theoremname}
\theoremstyle{remark}
\newtheorem*{rem*}{\protect\remarkname}
\theoremstyle{definition}
\newtheorem{defn}{\protect\definitionname}
\theoremstyle{plain}
\newtheorem{assumption}{\protect\assumptionname}
\theoremstyle{plain}
\newtheorem{lem}{\protect\lemmaname}
\theoremstyle{remark}
\newtheorem*{acknowledgement*}{\protect\acknowledgementname}

\usepackage{enumitem}
\setlist{nosep}
\usepackage{graphicx}

\makeatother

\providecommand{\acknowledgementname}{Acknowledgement}
\providecommand{\assumptionname}{Assumption}
\providecommand{\definitionname}{Definition}
\providecommand{\lemmaname}{Lemma}
\providecommand{\remarkname}{Remark}
\providecommand{\theoremname}{Theorem}

\begin{document}
\newcommand{\R}{\mathbb{R}}
\title{Nonlocal Fully Nonlinear Double Obstacle Problems}
\author{Mohammad Safdari$\,{}^{1}$}
\begin{abstract}
We prove the existence and $C^{1,\alpha}$ regularity of solutions
to nonlocal fully nonlinear elliptic double obstacle problems. We
also obtain boundary regularity for these problems. The obstacles
are assumed to be Lipschitz semi-concave/semi-convex functions, and
we do not require them to be $C^{1}$. Our approach is to adapt a
penalization method to be applicable to the setting of nonlocal equations
and their viscosity solutions.\medskip{}

\thanks{$^{1}\;$Department of Mathematical Sciences, Sharif University of
Technology, Tehran, Iran\protect \\
Email address: safdari@sharif.edu}
\end{abstract}

\maketitle

\section{Introduction}

In this paper we consider the existence and regularity of solutions
to the double obstacle problem 
\begin{equation}
\begin{cases}
\max\{\min\{-Iu-f,u-\psi^{-}\},u-\psi^{+}\}=0 & \textrm{in }U,\\
u=\varphi & \textrm{in }\mathbb{R}^{n}-U.
\end{cases}\label{eq: dbl obs}
\end{equation}
Here $I$ is a nonlocal elliptic operator, of which a prototypical
example is the fractional Laplacian 
\[
-(-\Delta)^{s}u(x)=c_{n,s}\int_{\mathbb{R}^{n}}\frac{u(x+y)+u(x-y)-2u(x)}{|y|^{n+2s}}\,dy.
\]
Nonlocal operators appear naturally in the study of discontinuous
stochastic processes as the jump part of their infinitesimal generator.
These operators have also been studied extensively in recent years
from the analytic viewpoint of integro-differential equations. The
foundational works of \citet{caffarelli2009regularity,caffarelli2011evans,caffarelli2011regularity}
paved the way and set the framework for such studies. They provided
an appropriate notion of ellipticity for nonlinear nonlocal equations,
and obtained their $C^{1,\alpha}$ regularity. They also obtained
Evans-Krylov-type $C^{2s+\alpha}$ regularity for convex equations.
An interesting property of their estimates is their uniformity as
$s\uparrow1$, which provides a new proof for the corresponding classical
estimates for local equations.

Free boundary problems involving nonlocal operators have also seen
many advancements. \citet{silvestre2007regularity} obtained $C^{1,\alpha}$
regularity of the obstacle problem for fractional Laplacian. \citet{caffarelli2008regularity}
proved the optimal $C^{1,s}$ regularity for this problem when the
obstacle is smooth enough. \citet*{bjorland2012nonlocal} studied
a double obstacle problem for the infinity fractional Laplacian which
appear in the study of a nonlocal version of the tug-of-war game.
\citet{korvenpaa2016obstacle} studied the obstacle problem for operators
of fractional $p$-Laplacian type. \citet{petrosyan2015optimal} considered
the obstacle problem for the fractional Laplacian with drift in the
subcritical regime $s\in(\frac{1}{2},1)$, and \citet{fernandez2018obstacle}
studied the critical case $s=\frac{1}{2}$. There has also been some
works on other types of nonlocal free boundary problems, like the
work of \citet{rodrigues2019nonlocal} on nonlocal linear variational
inequalities with constraint on the fractional gradient.

A major breakthrough in the study of nonlocal free boundary problems
came with the work of \citet{caffarelli2017obstacle}, in which they
obtained the regularity of the solution and of the free boundary of
the obstacle problem for a large class of nonlocal elliptic operators.
These problems appear naturally when considering optimal stopping
problems for Lévy processes with jumps, which arise for example as
option pricing models in mathematical finance. We should mention that
in their work, the boundary regularity results of \citet{ros2016boundary,ros2017boundary}
for nonlocal elliptic equations were also essential.

In this paper we prove $C^{1,\alpha}$ regularity of the double obstacle
problem (\ref{eq: dbl obs}) for a large class of nonlocal fully nonlinear
operators $I$. In contrast to \citep{caffarelli2017obstacle}, we
do not require the operator to be convex. We also allow less smooth
obstacles, and do not require them to be $C^{1}$. And in addition
to the interior regularity, we obtain the boundary regularity too.
Our estimates in this work are uniform as $s\uparrow1$; hence they
provide a new proof for the corresponding regularity results for local
double obstacle problems. We use these regularity results in an upcoming
work to study nonlocal equations with constraint on the (classical)
gradient. A particular consequence of those results is that the regularity
for nonlocal double obstacle problems can be proved for even less
smooth obstacles.

Let us also briefly mention some of the works on the (double) obstacle
problem for local equations, and local equations with gradient constraints.
The study of elliptic equations with gradient constraints was initiated
by \citet{MR529814} when he considered the problem 
\[
\max\{Lu-f,\;|Du|-g\}=0,
\]
where $L$ is a (local) linear elliptic operator. Equations of this
type stem from dynamic programming in a wide class of singular stochastic
control problems. Recently, there has been new interest in these types
of problems. \citet{Hynd} studied (local) fully nonlinear elliptic
equations with strictly convex gradient constraints of the form 
\[
\max\{F(x,D^{2}u)-f,\;H(Du)\}=0.
\]
Closely related to these problems are variational problems with gradient
constraints. An important example among them is the well-known elastic-plastic
torsion problem, which is the problem of minimizing the functional
$\int_{U}\frac{1}{2}|Dv|^{2}-v\,dx$ over the set 
\[
W_{B_{1}}:=\{v\in W_{0}^{1,2}(U):|Dv|\le1\textrm{ a.e.}\}.
\]
An interesting property of variational problems with gradient constraints
is that under mild conditions they are equivalent to double obstacle
problems. For example the minimizer of $\int_{U}G(Dv)\,dx$ over $W_{B_{1}}$,
also satisfies $-d\le v\le d$ and 
\[
\begin{cases}
-D_{i}(D_{i}G(Dv))=0 & \textrm{ in }\{-d<v<d\},\\
-D_{i}(D_{i}G(Dv))\le0 & \textrm{ a.e. on }\{v=d\},\\
-D_{i}(D_{i}G(Dv))\ge0 & \textrm{ a.e. on }\{v=-d\},
\end{cases}
\]
where $d$ is the distance to $\partial U$; see for example \citep{MR1}.
This problem can be more compactly written as 
\[
\max\{\min\{F(x,D^{2}v),v+d\},v-d\}=0,
\]
where $F(x,D^{2}v)=-D_{i}(D_{i}G(Dv))=-D_{ij}^{2}G(x)D_{ij}^{2}v$.

Variational problems with gradient constraints have also seen new
developments in recent years. \citet{MR2605868} investigated the
minimizers of some functionals subject to gradient constraints, arising
in the study of random surfaces. In their work, the functional is
allowed to have certain kinds of singularities. Also, the constraints
are given by convex polygons; so they are not strictly convex. In
\citep{SAFDARI202176} we have studied variational problems with non-strictly
convex gradient constraints, and we obtained their optimal $C^{1,1}$
regularity. This has been partly motivated by the above-mentioned
problem about random surfaces. There has also been similar interests
in elliptic equations with gradient constraints which are not strictly
convex. These problems emerge in the study of some singular stochastic
control problems appearing in financial models with transaction costs;
see for example \citep{barles1998option,possamai2015homogenization}.
In \citep{SAFDARI2021358} we extended the results of \citep{Hynd}
and proved the optimal regularity for (local) fully nonlinear elliptic
equations with non-strictly convex gradient constraints. Our approach
was to obtain a link between double obstacle problems and elliptic
equations with gradient constraints. This link has been well known
in the case where the double obstacle problem reduces to an obstacle
problem. However, we have shown that there is still a connection between
the two problems in the general case. In this approach, we also studied
(local) fully nonlinear double obstacle problems with singular obstacles.
Finally, let us also mention the recent works \citep{MR2989443,lee2019regularity}
on (local) double obstacle problems.

Now let us introduce the problem in more detail. First we recall some
of the definitions and conventions about nonlocal operators introduced
in \citep{caffarelli2009regularity}. Let 
\[
\delta u(x,y):=u(x+y)+u(x-y)-2u(x).
\]
A linear nonlocal operator is an operator of the form 
\[
Lu(x)=\int_{\mathbb{R}^{n}}\delta u(x,y)K(y)\,dy,
\]
where the kernel $K$ is a positive function which satisfies $K(-y)=K(y)$,
and 
\[
\int_{\mathbb{R}^{n}}\min\{1,|y|^{2}\}K(y)\,dy<\infty.
\]

We say a function $u$ belongs to $C^{1,1}(x)$ if there are quadratic
polynomials $P,Q$ such that $P(x)=u(x)=Q(x)$, and $P\le u\le Q$
on a neighborhood of $x$. A nonlocal operator $I$ is an operator
for which $Iu(x)$ is well-defined for bounded functions $u\in C^{1,1}(x)$,
and $Iu(\cdot)$ is a continuous function on an open set if $u$ is
$C^{2}$ over that open set. The operator $I$ is \textit{uniformly
elliptic} with respect to a family of linear operators $\mathcal{L}$
if for any bounded functions $u,v\in C^{1,1}(x)$ we have 
\[
M_{\mathcal{L}}^{-}(u-v)(x)\le Iu(x)-Iv(x)\le M_{\mathcal{L}}^{+}(u-v)(x),
\]
where the extremal Pucci-type operators $M_{\mathcal{L}}^{\pm}$
are defined as 
\[
M_{\mathcal{L}}^{-}u(x)=\inf_{L\in\mathcal{L}}Lu(x),\qquad\qquad M_{\mathcal{L}}^{+}u(x)=\sup_{L\in\mathcal{L}}Lu(x).
\]
Let us also note that $\pm M_{\mathcal{L}}^{\pm}$ are subadditive. 

An important family of linear operators is the class $\mathcal{L}_{0}$
of linear operators whose kernels are comparable with the kernel of
fractional Laplacian $-(-\Delta)^{s}$, i.e. 
\[
(1-s)\frac{\lambda}{|y|^{n+2s}}\le K(y)\le(1-s)\frac{\Lambda}{|y|^{n+2s}},
\]
where $0<s<1$ and $0<\lambda\le\Lambda$. It can be shown that in
this case the extremal operators $M_{\mathcal{L}_{0}}^{\pm}$, which
we will simply denote by $M^{\pm}$, are given by 
\begin{align*}
 & M^{+}u=(1-s)\int_{\mathbb{R}^{n}}\frac{\Lambda\delta u(x,y)^{+}-\lambda\delta u(x,y)^{-}}{|y|^{n+2s}}\,dy,\\
 & M^{-}u=(1-s)\int_{\mathbb{R}^{n}}\frac{\lambda\delta u(x,y)^{+}-\Lambda\delta u(x,y)^{-}}{|y|^{n+2s}}\,dy,
\end{align*}
where $r^{\pm}=\max\{\pm r,0\}$ for a real number $r$. Another important
class is $\mathcal{L}_{*}\subset\mathcal{L}_{0}$ which consists of
operators with homogeneous kernel, i.e. 
\[
K(y)=\frac{a(y/|y|)}{|y|^{n+2s}}\qquad\textrm{ with }\qquad(1-s)\lambda\le a(\cdot)\le(1-s)\Lambda.
\]
The class $\mathcal{L}_{*}$ consists of all infinitesimal generators
of stable Lévy processes belonging to $\mathcal{L}_{0}$, and appears
in the study of boundary regularity of nonlocal equations (see \citep{ros2016boundary,ros2017boundary}
for more details). We will denote the extremal operators $M_{\mathcal{L}_{*}}^{\pm}$
simply by $M_{*}^{\pm}$. Note that we have 
\[
M^{-}u\le M_{*}^{-}u\le M_{*}^{+}u\le M^{+}u.
\]

We will also only consider ``constant coefficient'' nonlocal operators,
i.e. we assume that $I$ is translation invariant: 
\[
I(\tau_{z}u)=\tau_{z}(Iu)
\]
for every $z$, where $\tau_{z}u(x):=u(x-z)$ is the translation operator.
In addition, without loss of generality we can assume that $I(0)=0$,
i.e. the action of $I$ on the constant function 0 is 0. Because by
translation invariance $I(0)$ is constant, and we can consider $I-I(0)$
instead of $I$.

Now let us state our main results. We denote the distance to $\partial U$
by $d(\cdot)=d(\cdot,\partial U)$.
\begin{thm}
\label{thm: Reg u}Suppose $I$ is a translation invariant operator
which is uniformly elliptic with respect to $\mathcal{L}_{0}$, and
$0<s_{0}<s<1$. Also, suppose $U,\varphi,\psi^{\pm}$ satisfy Assumption
\ref{assu: =00005Cpsi +-}, and $f:\R^{n}\to\R$ is a continuous function.
Then the double obstacle problem (\ref{eq: dbl obs}) has a viscosity
solution $u$, and 
\[
u\in C_{\mathrm{loc}}^{1,\alpha}(U)
\]
for some $\alpha>0$ depending only on $n,\lambda,\Lambda,s_{0}$.
And for an open subset $V\subset\subset U$ we have 
\[
\|u\|_{C^{1,\alpha}(\overline{V})}\le C(\|u\|_{L^{\infty}(\mathbb{R}^{n})}+\|f\|_{L^{\infty}(U)}+C_{V}),
\]
where $C$ depends only on $n,\lambda,\Lambda,s_{0}$, and $d(V,\partial U)$;
and $C_{V}$ depends only on these constants together with $\|\psi^{\pm}\|_{L^{\infty}(\mathbb{R}^{n})}$
and the semi-concavity norms of $\pm\psi^{\pm}$ on $V$.
\end{thm}
\begin{rem*}
Note that our estimate is uniform for $s>s_{0}$; so we can retrieve
the interior estimate for local double obstacle problems as $s\to1$.

Next we state our result about boundary regularity. As it is well
known (see \citep{ros2016boundary,ros2017boundary}), the boundary
regularity for nonlocal equations should be stated in terms of $u/d^{s}$,
instead of $u$.
\end{rem*}
\begin{thm}
\label{thm: Reg u bdry}Suppose in addition to the assumptions of
Theorem \ref{thm: Reg u}, $\psi^{\pm}$ are $C^{2}$ on a neighborhood
of $\partial U$ in $\R^{n}$, $\varphi$ is $C^{2}$ on $\R^{n}$,
and $I$ is elliptic with respect to $\mathcal{L}_{*}$. Then for
$x_{0}\in\partial U$ and $r$ small enough we have 
\[
(u-\varphi)/d^{s}\in C^{\tilde{\alpha}}(B_{r}(x_{0})),
\]
where $d$ is the distance to $\partial U$, and $\tilde{\alpha}>0$
depends only on $n,\lambda,\Lambda,s_{0}$. In addition, we have 
\[
\|(u-\varphi)/d^{s}\|_{C^{\tilde{\alpha}}(B_{r}(x_{0}))}\le\tilde{C}(\|u\|_{L^{\infty}(\mathbb{R}^{n})}+\|f\|_{L^{\infty}(U)}+C_{0}),
\]
where $\tilde{C}$ depends only on $n,\lambda,\Lambda,s,r,U$; and
$C_{0}$ depends only on $n,\lambda,\Lambda,s_{0}$, and $\varphi,\psi^{\pm}$%
\end{thm}
\begin{rem*}
The constant $\tilde{C}$ in the above estimate can be chosen uniformly
for $s>s_{0}$ if for example $\partial U\cap B_{r}(x_{0})$ is flat,
or $I$ is of the form $\inf_{\alpha}\sup_{\beta}L_{\alpha\beta}$
for a family of linear operators (see \citep{ros2016boundary} for
the exact statements and assumptions). However, to the best of author's
knowledge, currently there is no uniform in $s$ boundary estimate
for arbitrary elliptic operators with respect to $\mathcal{L}_{*}$
over arbitrary domains with $C^{2}$ boundary. 
\end{rem*}
\begin{rem*}
Note that unlike the local case (see \citep{SAFDARI2021358}), here
we need to require $\psi^{\pm}$ to be $C^{2}$ on a neighborhood
of $\partial U$ in $\R^{n}$, not merely on a neighborhood of $\partial U$
in $\overline{U}$. This extra restriction, which is not satisfied
by the obstacles arising in problems with gradient constraints, is
imposed on us by the nonlocal nature of the problem.
\end{rem*}
Finally let us provide a brief sketch of the proof. We first approximate
the obstacles with smoother obstacles, and use them to solve some
approximate double obstacle problems. To solve the approximate problems,
we use a penalization method, which is adapted to be applicable to
the setting of nonlocal equations and their viscosity solutions. Next
we obtain uniform estimates for the solutions to the approximate problems,
and show that they converge to a solution of (\ref{eq: dbl obs}).

\section{\label{sec: Proof}Proof of the Main Results}

Let us start with defining the notion of viscosity solutions of nonlocal
equations. We want to study the nonlocal double obstacle problem 
\[
\max\{\min\{-Iu-f,u-\psi^{-}\},u-\psi^{+}\}=0.
\]
We are also going to consider penalized versions of a nonlocal equation.
So we state the definition of viscosity solution in the more general
case of a nonlocal operator $\tilde{I}(x,u(x),u(\cdot))$ whose value
also depends on the pointwise values of $x,u(x)$ (see \citep{mou2017perron}
for more details). For example, in the case of the double obstacle
problem we have 
\[
-\tilde{I}(x,r,u(\cdot))=\max\{\min\{-Iu(x)-f(x),r-\psi^{-}(x)\},r-\psi^{+}(x)\},
\]
where we replaced $u(x)$ with $r$ to clarify the dependence of $\tilde{I}$
on its arguments.
\begin{defn}
An upper semi-continuous (USC) function $u$ is a \textit{viscosity
subsolution} of $-\tilde{I}\le0$ in $U$ if whenever $\phi$ is a
bounded $C^{2}$ function and $u-\phi$ has a maximum over $\R^{n}$
at $x_{0}\in U$ we have 
\[
-\tilde{I}(x_{0},u(x_{0}),\phi(\cdot))\le0.
\]
And a lower semi-continuous (LSC) function $u$ is a \textit{viscosity
supersolution} of $-\tilde{I}\ge0$ in $U$ if whenever $\phi$ is
a bounded $C^{2}$ function and $u-\phi$ has a minimum over $\R^{n}$
at $x_{0}\in U$ we have 
\[
-\tilde{I}(x_{0},u(x_{0}),\phi(\cdot))\ge0.
\]
A continuous function $u$ is a \textit{viscosity solution} of $-\tilde{I}=0$
in $U$ if it is a subsolution of $-\tilde{I}\le0$ and a supersolution
of $-\tilde{I}\ge0$ in $U$.
\end{defn}
Now let us state our assumptions about $U$, the obstacles $\psi^{\pm}$,
and $\varphi$. 
\begin{assumption}
\label{assu: =00005Cpsi +-}We assume that $U\subset\R^{n}$ is a
bounded open set with $C^{2}$ boundary, and $\psi^{\pm},\varphi:\R^{n}\to\R$
are bounded Lipschitz functions which satisfy
\begin{enumerate}
\item[\upshape{(a)}] For every $x,y\in\mathbb{R}^{n}$ we have 
\[
|\psi^{\pm}(x)-\psi^{\pm}(y)|\le C_{1}|x-y|.
\]
And similarly $|\varphi(x)-\varphi(y)|\le C_{1}|x-y|$.
\item[\upshape{(b)}] $\psi^{\pm}=\varphi$ on $\R^{n}-U$, and for all $x\in U$ we have
\begin{equation}
0<\psi^{+}(x)-\psi^{-}(x)\le2C_{1}d(x),\label{eq: psi- < psi+}
\end{equation}
where $d$ is the distance to $\partial U$.
\item[\upshape{(c)}] For every $x\in U$ and $|y|\le d(x)-\epsilon$ we have 
\begin{equation}
\pm\delta\psi^{\pm}(x,y)\leq C|y|^{2},\label{eq: bd D2 psi}
\end{equation}
where the constant $C$ depends only on $\epsilon$. In other words,
$\psi^{+},\psi^{-}$ are respectively semi-concave and semi-convex
on compact subsets of $U$.
\end{enumerate}
\end{assumption}
\begin{rem*}
A prototypical example of obstacles satisfying this assumption is
given by $\varphi=0$ and $\psi^{\pm}=\pm\rho$, where $\rho$ is
the distance to $\partial U$ with respect to some suitable norm.
These kinds of obstacles appear for example in the study of equations
with gradient constraints (see \citep{SAFDARI2021358,SAFDARI202176}). 
\end{rem*}
Let $\eta_{\varepsilon}$ be a standard mollifier whose support is
$\overline{B_{\varepsilon}(0)}$. Then we define
\begin{align}
 & \psi_{\varepsilon}^{+}(x):=(\eta_{\varepsilon}*\psi^{+})(x):=\int_{|y|\leq\varepsilon}\eta_{\varepsilon}(y)\psi^{+}(x-y)\,dy,\nonumber \\
 & \psi_{\varepsilon}^{-}(x):=(\eta_{\varepsilon}*\psi^{-})(x),\qquad\varphi_{\varepsilon}(x):=(\eta_{\varepsilon}*\varphi)(x).\label{eq: psi _e}
\end{align}
Note that $\psi_{\varepsilon}^{\pm},\varphi_{\varepsilon}$ are uniformly
bounded smooth functions on $\mathbb{R}^{n}$ which uniformly converge
to $\psi^{\pm},\varphi$. More explicitly, for every $x$ we have
\begin{align}
|\psi_{\varepsilon}^{\pm}(x)-\psi^{\pm}(x)|\leq\int_{|y|\leq\varepsilon}\eta_{\varepsilon}(y) & |\psi^{\pm}(x-y)-\psi^{\pm}(x)|\,dy\label{eq: psi_e - psi}\\
 & \leq\int_{|y|\leq\varepsilon}C_{1}|y|\eta_{\varepsilon}(y)\,dy\le C_{1}\varepsilon.\nonumber 
\end{align}
And similarly $|\varphi_{\varepsilon}(x)-\varphi(x)|\le C_{1}\varepsilon$.
In addition note that since $\psi^{\pm}$ are Lipschitz functions
we have $|D\psi^{\pm}|\le C_{1}$ a.e. Thus 
\begin{align}
|D\psi_{\varepsilon}^{\pm}(x)| & \leq\int_{|y|\leq\varepsilon}|\eta_{\varepsilon}(y)D\psi^{\pm}(x-y)|\,dy\label{eq: bd D psi _e}\\
 & =\int_{|y|\leq\varepsilon}\eta_{\varepsilon}(y)|D\psi^{\pm}(x-y)|\,dy\leq C_{1}\int_{|y|\leq\varepsilon}\eta_{\varepsilon}(y)\,dy\;=C_{1}.\nonumber 
\end{align}
Similarly $|D\varphi_{\varepsilon}|\le C_{1}$. Now let 
\begin{equation}
U_{\varepsilon}:=\{x\in\mathbb{R}^{n}:d(x,\overline{U})<\varepsilon\}.\label{eq: U_e}
\end{equation}
Then for $x\in U_{\varepsilon}$ we have $\psi_{\varepsilon}^{-}(x)<\psi_{\varepsilon}^{+}(x)$,
and for $x\in\R^{n}-U_{\varepsilon}$ we have $\psi_{\varepsilon}^{\pm}(x)=\varphi_{\varepsilon}(x)$.
In addition, since the distance function is $C^{2}$ on a tubular
neighborhood of $\partial U$ and its derivative is the normal to
the boundary, the $\partial U_{\varepsilon}$ is also $C^{2}$.
\begin{lem}
\label{lem: I psi_e bdd}Suppose that Assumption \ref{assu: =00005Cpsi +-}
holds. Also suppose $0<s_{0}<s<1$. Then for every bounded open set
$V\subset\subset U$ and $\varepsilon<\frac{1}{3}d(V,\partial U)$
there is a constant $C_{V}$ such that 
\[
\pm I\psi_{\varepsilon}^{\pm}\le C_{V}
\]
on $V$, where the constant $C_{V}$ depends only on $n,\lambda,\Lambda,s_{0}$,
and $d(V,\partial U)$.
\end{lem}
\begin{proof}
Let $\varepsilon_{0}:=\frac{1}{3}d(V,\partial U)$. Let $x\in V$.
Then for $|y|,|z|\le\varepsilon_{0}$ we have $d(x-z,\partial U)\ge2\varepsilon_{0}\ge|y|+\varepsilon_{0}$.
Hence by (\ref{eq: bd D2 psi}) we get 
\[
\pm\delta\psi_{\varepsilon}^{\pm}(x,y)=\pm\int_{|z|\le\varepsilon}\eta_{\varepsilon}(z)\delta\psi^{\pm}(x-z,y)\,dz\leq\int_{|z|\le\varepsilon}\eta_{\varepsilon}(z)C|y|^{2}\,dz=C|y|^{2},
\]
where $C$ depends only on $\varepsilon_{0}$. Hence we have 
\[
\delta\psi_{\varepsilon}^{+}(x,y)^{+}-\delta\psi_{\varepsilon}^{+}(x,y)^{-}=\delta\psi_{\varepsilon}^{+}(x,y)\le C|y|^{2}
\]
for $|y|\le\varepsilon_{0}$. We can make $C$ larger if necessary
so that for $|y|>\varepsilon_{0}$ we have $\delta\psi_{\varepsilon}^{+}(x,y)\le C\varepsilon_{0}^{2}$,
since $\psi_{\varepsilon}^{+}$ is bounded independently of $\varepsilon$.
(It suffices to take $C\ge4\|\psi^{+}\|_{L^{\infty}}/\varepsilon_{0}^{2}$.)
Thus by the ellipticity of $I$ we get 
\begin{align*}
I\psi_{\varepsilon}^{+}(x) & \le I0(x)+M^{+}\psi_{\varepsilon}^{+}(x)\\
 & =0+(1-s)\int_{\mathbb{R}^{n}}\frac{\Lambda\delta\psi_{\varepsilon}^{+}(x,y)^{+}-\lambda\delta\psi_{\varepsilon}^{+}(x,y)^{-}}{|y|^{n+2s}}\,dy\\
 & \le(1-s)\int_{\mathbb{R}^{n}}\frac{(\Lambda+\lambda)C\min\{\varepsilon_{0}^{2},|y|^{2}\}}{|y|^{n+2s}}\,dy\\
 & =(1-s)(\Lambda+\lambda)C\int_{\mathbb{S}^{n-1}}\int_{0}^{\infty}\frac{\min\{\varepsilon_{0}^{2},r^{2}\}}{r^{n+2s}}\,r^{n-1}drdS\\
 & =(1-s)\hat{C}\int_{0}^{\infty}\min\{\varepsilon_{0}^{2},r^{2}\}\,r^{-1-2s}dr=\frac{\hat{C}}{2s}\varepsilon_{0}^{2-2s}\le\frac{\hat{C}\varepsilon_{0}^{2-2s}}{2s_{0}}=:C_{V}<\infty,
\end{align*}
as desired. We can similarly obtain a uniform bound for $-I\psi_{\varepsilon}^{-}$.
\end{proof}
\begin{rem*}
Note that if in the inequality $\delta\psi_{\varepsilon}^{+}(x,y)\le C|y|^{2}$
we divide by $|y|^{2}$ and let $|y|\to0$ (without changing the direction
of $y$) we get $D_{\hat{y}\hat{y}}^{2}\psi_{\varepsilon}^{+}(x)\le C$,
where $\hat{y}=y/|y|$ is the unit vector in the direction of $y$.
Conversely, a bound for the second derivative of a $C^{2}$ function
like $\psi_{\varepsilon}^{+}$ implies a bound for $\delta\psi_{\varepsilon}^{+}$,
since we can easily see that 
\[
\delta\psi_{\varepsilon}^{+}(x,y)=|y|^{2}\int_{0}^{1}\int_{-1}^{1}tD_{\hat{y}\hat{y}}^{2}\psi_{\varepsilon}^{+}(x+sty)\,dsdt.
\]
\end{rem*}

\begin{lem}
\label{lem: Reg u_e}Suppose $I$ is a translation invariant operator
which is uniformly elliptic with respect to $\mathcal{L}_{0}$.
Also, suppose $U,\varphi,\psi^{\pm}$ satisfy Assumption \ref{assu: =00005Cpsi +-},
and $f:\R^{n}\to\R$ is a continuous function. Then the double obstacle
problem 
\begin{equation}
\begin{cases}
\max\{\min\{-Iu_{\varepsilon}-f,u_{\varepsilon}-\psi_{\varepsilon}^{-}\},u_{\varepsilon}-\psi_{\varepsilon}^{+}\}=0 & \textrm{in }U_{\varepsilon},\\
u_{\varepsilon}=\varphi_{\varepsilon} & \textrm{in }\mathbb{R}^{n}-U_{\varepsilon}.
\end{cases}\label{eq: dbl obs u_e}
\end{equation}
has a viscosity solution $u_{\varepsilon}$, and 
\[
u_{\varepsilon}\in C_{\mathrm{loc}}^{1,\alpha}(U_{\varepsilon}).
\]
\end{lem}
\begin{proof}
Fix $\varepsilon>0$. For $\delta>0$ let $\beta_{\delta}$ be a smooth
increasing function that vanishes on $(-\infty,0]$ and equals $\frac{1}{\delta}t$
for $t\ge\delta$. Then the equation 
\begin{equation}
\begin{cases}
-I\tilde{u}_{\delta}-f-\beta_{\delta}(\psi_{\varepsilon}^{-}-\tilde{u}_{\delta})+\beta_{\delta}(\tilde{u}_{\delta}-\psi_{\varepsilon}^{+})=0 & \textrm{in }U_{\varepsilon},\\
\tilde{u}_{\delta}=\varphi_{\varepsilon} & \textrm{in }\mathbb{R}^{n}-U_{\varepsilon},
\end{cases}\label{eq: u-e,d}
\end{equation}
has a viscosity solution (see Theorem 5.6 of \citep{mou2017perron}).
(Here we are also using the fact that $\partial U_{\varepsilon}$
is $C^{2}$.) To simplify the notation we set 
\[
\tilde{u}=\tilde{u}_{\delta},\qquad\qquad\beta=\beta_{\delta}.
\]

Now let us show that 
\[
\|\beta(\pm(\tilde{u}-\psi_{\varepsilon}^{\pm}))\|_{L^{\infty}(U_{\varepsilon})}\le\tilde{C}_{0},
\]
where $\tilde{C}_{0}$ is independent of $\delta$. Note that $\beta(\pm(\tilde{u}-\psi_{\varepsilon}^{\pm}))$
is zero on $\R^{n}-U_{\varepsilon}$. So assume that $\beta(\pm(\tilde{u}-\psi_{\varepsilon}^{\pm}))$
attains its positive maximum at $x_{0}\in U_{\varepsilon}$. Let us
consider $\beta(\tilde{u}-\psi_{\varepsilon}^{+})$; the other case
is similar. Since $\beta$ is increasing, $\tilde{u}-\psi_{\varepsilon}^{+}$
has a positive maximum at $x_{0}$ too. Therefore by the definition
of viscosity solution we have 
\[
-I\psi_{\varepsilon}^{+}(x_{0})-f(x_{0})-\beta(\psi_{\varepsilon}^{-}(x_{0})-\tilde{u}(x_{0}))+\beta(\tilde{u}(x_{0})-\psi_{\varepsilon}^{+}(x_{0}))\le0.
\]
So at $x_{0}$ we have 
\[
-I\psi_{\varepsilon}^{+}(x_{0})-f(x_{0})\le\beta(\psi_{\varepsilon}^{-}-\tilde{u})-\beta(\tilde{u}-\psi_{\varepsilon}^{+})=-\beta(\tilde{u}-\psi_{\varepsilon}^{+}),
\]
since by our assumption $\psi_{\varepsilon}^{-}(x_{0})<\psi_{\varepsilon}^{+}(x_{0})<\tilde{u}(x_{0})$.
Thus $\beta(\tilde{u}-\psi_{\varepsilon}^{+})\le I\psi_{\varepsilon}^{+}+f$
at $x_{0}$. Therefore $\beta(\tilde{u}-\psi_{\varepsilon}^{+})$
is bounded independently of $\delta$, as desired, because $I\psi_{\varepsilon}^{+},f$
are continuous functions.

The bound $\beta(\pm(\tilde{u}-\psi_{\varepsilon}^{\pm}))\le\tilde{C}_{0}$
and the definition of $\beta$ imply that 
\begin{equation}
\tilde{u}-\psi_{\varepsilon}^{+}\le\delta(\tilde{C}_{0}+1),\qquad\qquad\psi_{\varepsilon}^{-}-\tilde{u}\le\delta(\tilde{C}_{0}+1).\label{eq: 1 in Reg u_e}
\end{equation}
This also shows that $\tilde{u}$ is uniformly bounded independently
of $\delta$. In addition, we can choose $\tilde{C}_{0}$ large enough
so that $|f|\le\tilde{C}_{0}$. Then from the equation (\ref{eq: u-e,d})
we conclude that 
\[
-3\tilde{C}_{0}\le I\tilde{u}\le3\tilde{C}_{0}
\]
in the viscosity sense.

Thus by Theorem 4.1 of \citep{kriventsov2013c} if $B_{r}(x_{0})\subset U_{\varepsilon}$
we have 
\[
\|\tilde{u}\|_{C^{1,\alpha}(B_{r/2}(x_{0}))}\le\frac{C}{r^{1+\alpha}}(\|\tilde{u}\|_{L^{\infty}(\mathbb{R}^{n})}+3\tilde{C}_{0}r^{2s}),
\]
where $C,\alpha$ depend only on $n,s,\lambda,\Lambda$ (actually,
$C,\alpha$ can be chosen uniformly for $s>s_{0}$). For simplicity
we are assuming that $r\le1$. %
(Note that by considering the scaled operator $I_{r}v(\cdot)=r^{2s}Iv(\frac{\cdot}{r})$,
which has the same ellipticity constants $\lambda,\Lambda$ as $I$,
and using the translation invariance of $I$, we have obtained the
estimate on the domain $B_{r/2}(x_{0})$ instead of $B_{1/2}(0)$.)%
{} Then we can cover an open subset $V\subset\subset U_{\varepsilon}$
with finitely many open balls contained in $U_{\varepsilon}$ and
obtain 
\[
\|\tilde{u}\|_{C^{1,\alpha}(\overline{V})}\le C(\|\tilde{u}\|_{L^{\infty}(\mathbb{R}^{n})}+3\tilde{C}_{0}),
\]
where $C$ depends only on $n,s,\lambda,\Lambda$, and $d(V,\partial U_{\varepsilon})$.
Therefore $\tilde{u}=\tilde{u}_{\delta}$ is bounded in $C^{1,\alpha}(\overline{V})$
independently of $\delta$, due to the uniform boundedness of $\tilde{u}$
in $L^{\infty}$.

Now, we choose a decreasing sequence $\delta_{j}\to0$, and let 
\[
V_{j}:=\{x\in U_{\varepsilon}:d(x,\partial U_{\varepsilon})>\delta_{j}\}.
\]
For convenience we also set $\tilde{u}_{j}:=\tilde{u}_{\delta_{j}}$.
Consider the sequence $\tilde{u}_{j}|_{V_{1}}$. Then $\|\tilde{u}_{j}\|_{C^{1,\alpha}(\overline{V}_{1})}$
is bounded independently of $j$. Hence there is a subsequence of
$\tilde{u}_{j}$'s, which we denote by $\tilde{u}_{j_{1}}$, that
is convergent in $C^{1}$ norm to a function $u_{\varepsilon,1}$
in $C^{1,\alpha}(\overline{V}_{1})$. Now we can repeat this process
with $\tilde{u}_{j_{1}}|_{V_{2}}$ and get a function $u_{\varepsilon,2}$
in $C^{1,\alpha}(\overline{V}_{2})$, which agrees with $u_{\varepsilon,1}$
on $V_{1}$. Continuing this way with subsequences $\tilde{u}_{j_{l}}$
for each positive integer $l$, we can finally construct a function
$u_{\varepsilon}$ in $C_{\mathrm{loc}}^{1,\alpha}(U_{\varepsilon})$.
Note that the diagonal sequence $\tilde{u}_{l_{l}}$, which we denote
by $\tilde{u}_{l}$, converges pointwise to $u_{\varepsilon}$ on
$U_{\varepsilon}$, and converges uniformly to $u_{\varepsilon}$
on compact subsets of $U_{\varepsilon}$. Also note that if we let
$\delta_{l}\to0$ in (\ref{eq: 1 in Reg u_e}) we get $\psi_{\varepsilon}^{-}\leq u_{\varepsilon}\leq\psi_{\varepsilon}^{+}$.
Consequently, as we approach $\partial U_{\varepsilon}$, $u_{\varepsilon}$
converges to $\varphi_{\varepsilon}$. We extend $u_{\varepsilon}$
to all of $\R^{n}$ by setting it equal to $\varphi_{\varepsilon}$
on $\R^{n}-U_{\varepsilon}$. Note that $u_{\varepsilon}$ is a continuous
function.

Let us show that $\tilde{u}_{l}$ converges uniformly to $u_{\varepsilon}$
on $\R^{n}$. Suppose $\epsilon$ is given and we want to show that
$\sup_{\R^{n}}|\tilde{u}_{l}-u_{\varepsilon}|<\epsilon$ for large
enough $l$. We know that $\tilde{u}_{l}=\varphi_{\varepsilon}=u_{\varepsilon}$
on $\R^{n}-U_{\varepsilon}$. And by (\ref{eq: 1 in Reg u_e}), the
fact that $\psi_{\varepsilon}^{-}\leq u_{\varepsilon}\leq\psi_{\varepsilon}^{+}$,
and (\ref{eq: psi- < psi+}) we have 
\[
|\tilde{u}_{l}-u_{\varepsilon}|\le\psi_{\varepsilon}^{+}-\psi_{\varepsilon}^{-}+\delta_{l}(\tilde{C}_{0}+1)\le2C_{1}d(\cdot)+\delta_{l}(\tilde{C}_{0}+1).
\]
Now let $V\subset\subset U$ be such that for $x\in U-V$ we have
$2C_{1}d(x)<\epsilon/2$. Then if $l$ is large enough so that $\sup_{V}|\tilde{u}_{l}-u_{\varepsilon}|<\epsilon$
and $\delta_{l}(\tilde{C}_{0}+1)<\epsilon/2$, we get the desired.

Finally, let us show that $u_{\varepsilon}$ satisfies the double
obstacle problem (\ref{eq: dbl obs u_e}). Suppose $\phi$ is a bounded
$C^{2}$ function and $u_{\varepsilon}-\phi$ has a maximum over $\R^{n}$
at $x_{0}\in U$. Let us first consider the case where $u_{\varepsilon}-\phi$
has a strict maximum at $x_{0}$. We must show that at $x_{0}$ we
have 
\begin{equation}
\max\{\min\{-I\phi(x_{0})-f,u_{\varepsilon}-\psi_{\varepsilon}^{-}\},u_{\varepsilon}-\psi_{\varepsilon}^{+}\}\le0.\label{eq: 2 in Reg u_e}
\end{equation}
Now we know that $\tilde{u}_{l}-\phi$ takes its global maximum at
a point $x_{l}$ where $x_{l}\to x_{0}$; because $\tilde{u}_{l}$
uniformly converges to $u_{\varepsilon}$ on $\R^{n}$.

We also know that $\psi_{\varepsilon}^{-}\leq u_{\varepsilon}\leq\psi_{\varepsilon}^{+}$.
If $\psi_{\varepsilon}^{-}(x_{0})=u_{\varepsilon}(x_{0})$ then (\ref{eq: 2 in Reg u_e})
holds trivially. So suppose $\psi_{\varepsilon}^{-}(x_{0})<u_{\varepsilon}(x_{0})$.
Then for large $l$ we have $\psi_{\varepsilon}^{-}(x_{l})<\tilde{u}_{l}(x_{l})$.
On the other hand, since $\tilde{u}_{l}$ is a viscosity solution
of the equation (\ref{eq: u-e,d}), at $x_{l}$ we have 
\[
-I\phi(x_{l})-f-\beta_{\delta_{l}}(\psi_{\varepsilon}^{-}-\tilde{u}_{l})+\beta_{\delta_{l}}(\tilde{u}_{l}-\psi_{\varepsilon}^{+})\le0.
\]
But $\beta_{\delta_{l}}(\psi_{\varepsilon}^{-}-\tilde{u}_{l})=0$
at $x_{l}$, so 
\[
-I\phi(x_{l})-f\le-\beta_{\delta_{l}}(\tilde{u}_{l}-\psi_{\varepsilon}^{+})\le0.
\]
Thus by letting $l\to\infty$ and using the continuity of $I\phi$
and $f$ we see that (\ref{eq: 2 in Reg u_e}) holds in this case
too. 

Now if the maximum of $u_{\varepsilon}-\phi$ at $x_{0}$ is not strict,
we can approximate $\phi$ with $\phi_{\epsilon}=\phi+\epsilon\tilde{\phi}$,
where $\tilde{\phi}$ is a bounded $C^{2}$ functions which vanishes
at $x_{0}$ and is positive elsewhere. Then, as we have shown, when
$\psi_{\varepsilon}^{-}(x_{0})<u_{\varepsilon}(x_{0})$ we have $-I\phi_{\epsilon}(x_{0})\le f(x_{0})$.
Hence by the ellipticity of $I$ we get 
\[
-I\phi(x_{0})\le M^{+}(\epsilon\tilde{\phi})(x_{0})-I\phi_{\epsilon}(x_{0})\le\epsilon M^{+}\tilde{\phi}(x_{0})+f(x_{0})\underset{\epsilon\to0}{\longrightarrow}f(x_{0}),
\]
as desired. Similarly, we can show that when $u_{\varepsilon}-\phi$
has a global minimum at $x_{0}\in U$ we have 
\[
\max\{\min\{-I\phi(x_{0})-f,u_{\varepsilon}-\psi_{\varepsilon}^{-}\},u_{\varepsilon}-\psi_{\varepsilon}^{+}\}\ge0.
\]
Therefore $u_{\varepsilon}$ is a viscosity solution of equation (\ref{eq: dbl obs u_e})
as desired.
\end{proof}
Finally we have
\begin{proof}[\textbf{Proof of Theorem \ref{thm: Reg u}}]
Let $u_{\varepsilon}$ be as in Lemma \ref{lem: Reg u_e}. Let us
first show that for every bounded open set $V\subset\subset U$ and
every small enough $\varepsilon$ we have 
\begin{equation}
-C_{V}-\|f\|_{L^{\infty}(U)}\le Iu_{\varepsilon}\le\|f\|_{L^{\infty}(U)}+C_{V}\label{eq: I(u_e) bdd}
\end{equation}
in the viscosity sense on $V$, with the constant $C_{V}$ given in
Lemma \ref{lem: I psi_e bdd}.

Suppose $\phi$ is a bounded $C^{2}$ function and $u_{\varepsilon}-\phi$
has a maximum over $\R^{n}$ at $x_{0}\in V$. We must show that 
\[
-I\phi(x_{0})\le\|f\|_{L^{\infty}(U)}+C_{V}.
\]
We can assume that $u_{\varepsilon}(x_{0})-\phi(x_{0})=0$ without
loss of generality, since we can consider $\phi+c$ instead of $\phi$
without changing $I$ (because $M^{\pm}(c)=0$). So we can assume
that $u_{\varepsilon}-\phi\le0$, or $u_{\varepsilon}\le\phi$. We
also know that at $x_{0}$ we have 
\[
\max\{\min\{-I\phi(x_{0})-f,u_{\varepsilon}-\psi_{\varepsilon}^{-}\},u_{\varepsilon}-\psi_{\varepsilon}^{+}\}\le0,
\]
since $u_{\varepsilon}$ is a viscosity solution of (\ref{eq: dbl obs u_e}).
In addition remember that $\psi_{\varepsilon}^{-}\leq u_{\varepsilon}\leq\psi_{\varepsilon}^{+}$.
Now if $\psi_{\varepsilon}^{-}(x_{0})<u_{\varepsilon}(x_{0})\le\psi_{\varepsilon}^{+}(x_{0})$
then we must have $-I\phi(x_{0})\le f(x_{0})\le\|f\|_{L^{\infty}(U)}$.
And if $u_{\varepsilon}(x_{0})=\psi_{\varepsilon}^{-}(x_{0})$ then
$\phi$ is also touching $\psi_{\varepsilon}^{-}$ from above at $x_{0}$,
since $\psi_{\varepsilon}^{-}\le u_{\varepsilon}\le\phi$. But by
Lemma \ref{lem: I psi_e bdd} we know that $-I\psi_{\varepsilon}^{-}\le C_{V}$
on $V$. So we must have 
\[
-I\phi(x_{0})\le-I\psi_{\varepsilon}^{-}(x_{0})+M^{+}(\psi_{\varepsilon}^{-}-\phi)(x_{0})\le C_{V},
\]
since it is easy to see that $M^{+}(\psi_{\varepsilon}^{-}-\phi)(x_{0})\le0$
as $L(\psi_{\varepsilon}^{-}-\phi)(x_{0})\le0$ for any linear operator
$L$. Thus in either case we must have $-I\phi(x_{0})\le\|f\|_{L^{\infty}(U)}+C_{V}$,
and therefore $-Iu_{\varepsilon}\le\|f\|_{L^{\infty}(U)}+C_{V}$ in
the viscosity sense. (Heuristically, note that on the contact set
$\{u_{\varepsilon}=\psi_{\varepsilon}^{+}\}$ although a priori we
do not have a lower bound for the second derivative of $\psi_{\varepsilon}^{+}$,
and hence an upper bound for $-I\psi_{\varepsilon}^{+}$, we can obtain
the desired bound for $-Iu_{\varepsilon}$ from the equation.) The
lower bound for $-Iu_{\varepsilon}$ can be shown to hold similarly.

Thus by Theorem 4.1 of \citep{kriventsov2013c}, and similarly to
the proof of Lemma \ref{lem: Reg u_e}, we can show that there is
$\alpha$ depending only on $n,\lambda,\Lambda,s_{0}$, such that
for an open subset $V\subset\subset U$ we have 
\[
\|u_{\varepsilon}\|_{C^{1,\alpha}(\overline{V})}\le C(\|u_{\varepsilon}\|_{L^{\infty}(\mathbb{R}^{n})}+\|f\|_{L^{\infty}(U)}+C_{V}),
\]
where $C$ depends only on $n,\lambda,\Lambda,s_{0}$, and $d(V,\partial U)$.
In particular $C$ does not depend on $\varepsilon$. Therefore $u_{\varepsilon}$
is bounded in $C^{1,\alpha}(\overline{V})$ independently of $\varepsilon$,
because $\|u_{\varepsilon}\|_{L^{\infty}}$ is bounded by $\|\psi_{\varepsilon}^{\pm}\|_{L^{\infty}}$,
and $\|\psi_{\varepsilon}^{\pm}\|_{L^{\infty}}$ are uniformly bounded
by $\|\psi^{\pm}\|_{L^{\infty}}$. Now we choose a decreasing sequence
$\varepsilon_{k}\to0$, and a sequence $\overline{V}_{k}\subset V_{k+1}$
such that $U=\bigcup_{k\ge1}V_{k}$. We also assume that $\varepsilon_{k}<\frac{1}{3}d(V_{k},\partial U)$.
For convenience we denote $u_{\varepsilon_{k}},\psi_{\varepsilon_{k}}^{\pm}$
by $u_{k},\psi_{k}^{\pm}$. Consider the sequence $u_{k}|_{V_{1}}$.
Then $\|u_{k}\|_{C^{1,\alpha}(\overline{V}_{1})}$ is bounded independently
of $k$. Hence there is a subsequence of $u_{k}$'s, which we denote
by $u_{k_{1}}$, that is convergent in $C^{1}$ norm to a function
in $C^{1,\alpha}(\overline{V}_{1})$. Now we can repeat this process
with $u_{k_{1}}|_{V_{2}}$ %
and get a function in $C^{1,\alpha}(\overline{V}_{2})$, which agrees
with the previous limit on $V_{1}$. Continuing this way with subsequences
$u_{k_{l}}$ for each positive integer $l$, we can finally construct
a function $u$ in $C_{\textrm{loc}}^{1,\alpha}(U)$. 

Note that the diagonal sequence $u_{l_{l}}$, which we denote by $u_{l}$,
converges pointwise to $u$ on $U$, and converges uniformly to $u$
on compact subsets of $U$. It is also obvious that $\psi^{-}\le u\le\psi^{+}$,
since $\psi_{l}^{-}\le u_{l}\le\psi_{l}^{+}$ for every $l$. Consequently,
as we approach $\partial U$, $u$ converges to $\varphi$. We extend
$u$ to all of $\R^{n}$ by setting it equal to $\varphi$ on $\R^{n}-U$.
Note that $u$ is a continuous function. Let us also show that $u_{l}$
converges uniformly to $u$ on $\R^{n}$. Suppose $\epsilon$ is given
and we want to show that $\sup_{\R^{n}}|u_{l}-u|<\epsilon$ for large
enough $l$. Since $u,u_{l}$ are between their corresponding obstacles,
by using (\ref{eq: psi- < psi+}),(\ref{eq: psi_e - psi}) we get
\begin{align*}
|u_{l}-u| & \le\max\{|\psi_{l}^{+}-\psi^{-}|,|\psi_{l}^{-}-\psi^{+}|\}\\
 & \le|\psi^{+}-\psi^{-}|+\max\{|\psi_{l}^{+}-\psi^{+}|,|\psi_{l}^{-}-\psi^{-}|\}\\
 & \le\begin{cases}
2C_{1}d(\cdot)+C_{1}\varepsilon_{l} & \textrm{in }U,\\
C_{1}\varepsilon_{l} & \textrm{in }\R^{n}-U,
\end{cases}
\end{align*}
because $|\psi^{+}-\psi^{-}|=0$ outside of $U$. Now let $V\subset\subset U$
be such that for $x\in U-V$ we have $2C_{1}d(x)<\epsilon/2$. Then
if $l$ is large enough so that $\sup_{V}|\tilde{u}_{l}-u_{\varepsilon}|<\epsilon$
and $C_{1}\varepsilon_{l}<\epsilon/2$, we get the desired.

Finally, due to the stability of viscosity solutions, $u$ must satisfy
the double obstacle problem (\ref{eq: dbl obs}). (We can also show
this similarly to the proof of Lemma \ref{lem: Reg u_e}.) 
\end{proof}
\begin{proof}[\textbf{Proof of Theorem \ref{thm: Reg u bdry}}]
We are assuming that $\psi^{\pm}$ are $C^{2}$ on a neighborhood
$W$ of $\partial U$. Let $W_{1}\subset\subset W$ be a smaller
neighborhood of $\partial U$. Let $0\le\zeta\le1$ be a $C^{\infty}$
function with compact support in $W$ which equals $1$ on $W_{1}$.
Set 
\[
\hat{\psi}_{\varepsilon}^{\pm}:=\zeta\psi^{\pm}+(1-\zeta)\psi_{\varepsilon}^{\pm}
\]
for $\varepsilon$ small enough. Note that $\hat{\psi}_{\varepsilon}^{\pm}$
are $C^{2}$ on a neighborhood of $\overline{U}$, and agree on $\partial U$.
Also, $\hat{\psi}_{\varepsilon}^{\pm}$ uniformly converge to $\psi^{\pm}$
as $\varepsilon\to0$. In fact, by (\ref{eq: psi_e - psi}) we can
easily see that $|\hat{\psi}_{\varepsilon}^{\pm}(x)-\psi^{\pm}(x)|\le C_{1}\varepsilon$.
It is obvious that $\hat{\psi}_{\varepsilon}^{-}=\psi^{-}<\psi^{+}=\hat{\psi}_{\varepsilon}^{+}$
on $W_{1}$. Let $W_{2}\subset\subset W_{1}$ be a yet smaller neighborhood
of $\partial U$. Then due to compactness, on $U-W_{2}$ we have $\psi^{+}-\psi^{-}\ge c>0$
for some $c$. Hence for $x\in U-W_{1}$ and small enough $\varepsilon$
we have 
\begin{align*}
\psi_{\varepsilon}^{+}(x)-\psi_{\varepsilon}^{-}(x)=\int_{|y|\leq\varepsilon}\eta_{\varepsilon}(y) & [\psi^{+}(x-y)-\psi^{-}(x-y)]\,dy\\
 & \qquad\qquad\ge c\int_{|y|\leq\varepsilon}\eta_{\varepsilon}(y)\,dy=c,
\end{align*}
and thus we get 
\[
\hat{\psi}_{\varepsilon}^{+}-\hat{\psi}_{\varepsilon}^{-}=\zeta(\psi^{+}-\psi^{-})+(1-\zeta)(\psi_{\varepsilon}^{+}-\psi_{\varepsilon}^{-})\ge c(\zeta+1-\zeta)=c>0.
\]
Therefore we have $\hat{\psi}_{\varepsilon}^{-}<\hat{\psi}_{\varepsilon}^{+}$
on $U$. 

Furthermore, note that by the proof of Lemma \ref{lem: I psi_e bdd}
and the remark after it, $\pm D^{2}\psi_{\varepsilon}^{\pm}$ is uniformly
bounded from above on $U-W_{1}=U\cap\{1-\zeta>0\}$. Hence 
\begin{align*}
\pm D^{2}\hat{\psi}_{\varepsilon}^{\pm} & =\pm\zeta D^{2}\psi^{\pm}\pm2D\zeta D\psi^{\pm}\pm\psi^{\pm}D^{2}\zeta\\
 & \qquad\qquad\pm(1-\zeta)D^{2}\psi_{\varepsilon}^{\pm}\mp2D\zeta D\psi_{\varepsilon}^{\pm}\mp\psi_{\varepsilon}^{\pm}D^{2}\zeta
\end{align*}
is uniformly bounded from above on $U\cup W_{1}\supset\overline{U}$.
Because $D^{2}\psi^{\pm}$ is bounded on the support of $\zeta$,
and $D\psi_{\varepsilon}^{\pm}$ is uniformly bounded due to (\ref{eq: bd D psi _e}).
So similarly to the proof of Lemma \ref{lem: I psi_e bdd} and employing
the remark after it, we can easily show that $\pm I\hat{\psi}_{\varepsilon}^{\pm}\le C_{U}$
on $U$, for some constant $C_{U}$ independent of $\varepsilon$.

Now we can repeat the construction of $u_{\varepsilon}$ with $\hat{\psi}_{\varepsilon}^{\pm}$
instead of $\psi_{\varepsilon}^{\pm}$. Note that in this case we
have $U_{\varepsilon}=U$ for every $\varepsilon$. So we can find
viscosity solutions $u_{\varepsilon}\in C_{\mathrm{loc}}^{1,\alpha}(U)$
of 
\begin{equation}
\begin{cases}
\max\{\min\{-Iu_{\varepsilon}-f,u_{\varepsilon}-\hat{\psi}_{\varepsilon}^{-}\},u_{\varepsilon}-\hat{\psi}_{\varepsilon}^{+}\}=0 & \textrm{in }U,\\
u_{\varepsilon}=\varphi & \textrm{in }\mathbb{R}^{n}-U.
\end{cases}\label{eq: dbl obs u^}
\end{equation}
Then similarly to the proof of Theorem \ref{thm: Reg u} we can show
that for every small enough $\varepsilon$ we have 
\begin{equation}
-C_{U}-\|f\|_{L^{\infty}(U)}\le Iu_{\varepsilon}\le\|f\|_{L^{\infty}(U)}+C_{U}\label{eq: I(u_e) bdd^}
\end{equation}
in the viscosity sense on $U$. Then, again similarly to the proof
of Theorem \ref{thm: Reg u}, we can construct a function $u$ in
$C_{\textrm{loc}}^{1,\alpha}(U)$ satisfying $\psi^{-}\le u\le\psi^{+}$,
which we extend to all of $\R^{n}$ by setting it equal to $\varphi$
on $\R^{n}-U$. In addition, due to the stability of viscosity solutions,
$u$ satisfies the double obstacle problem (\ref{eq: dbl obs}). 

Now let $x_{0}\in\partial U$ and suppose $B_{r}(x_{0})\subset\subset W_{1}$.
Also set $v_{\varepsilon}:=u_{\varepsilon}-\varphi$. Let us show
that 
\begin{equation}
\begin{cases}
M_{*}^{+}v_{\varepsilon}\ge-\|f\|_{L^{\infty}(U)}-C_{0} & \textrm{in }B_{r}(x_{0})\cap U\\
M_{*}^{-}v_{\varepsilon}\le\|f\|_{L^{\infty}(U)}+C_{0} & \textrm{in }B_{r}(x_{0})\cap U
\end{cases}\label{eq: Mv < >}
\end{equation}
in the viscosity sense, for some constant $C_{0}$ independent of
$\varepsilon$. Suppose $\phi$ is a bounded $C^{2}$ function and
$v_{\varepsilon}-\phi$ has a maximum over $\R^{n}$ at $x\in B_{r}(x_{0})\cap U$.
We must show that 
\[
-M_{*}^{+}\phi(x)\le\|f\|_{L^{\infty}(U)}+C_{0}.
\]
Now note that $u_{\varepsilon}-(\varphi+\phi)$ has a maximum over
$\R^{n}$ at $x\in B_{r}(x_{0})\cap U$, and $\phi+\varphi$ is a
bounded $C^{2}$ function. Hence by (\ref{eq: I(u_e) bdd^}) we must
have 
\[
-I(\phi+\varphi)(x)\le\|f\|_{L^{\infty}(U)}+C_{U}.
\]
Then by the subadditivity of $M_{*}^{+}$ and ellipticity of $I$
with respect to $\mathcal{L}_{*}$, at $x$ we have 
\[
-M_{*}^{+}\phi-M_{*}^{+}\varphi\le-M_{*}^{+}(\phi+\varphi)\le-I(\phi+\varphi)\le\|f\|_{L^{\infty}(U)}+C_{U}.
\]
However, since $\varphi$ is $C^{2}$, similarly to the proof of
Lemma \ref{lem: I psi_e bdd} we can show that $M_{*}^{+}\varphi\le M^{+}\varphi\le C$
on $B_{r}(x_{0})$. Hence we get the desired bound for $-M_{*}^{+}\phi(x)$.
We can similarly prove the other inequality in (\ref{eq: Mv < >}). 

Next note that in addition to (\ref{eq: Mv < >}), $v_{\varepsilon}=0$
in $B_{r}(x_{0})-U$. Thus by Theorem 1.5 of \citep{ros2017boundary}
(also see \citep{ros2016boundary}) there is $\tilde{\alpha}>0$ depending
only on $n,\lambda,\Lambda,s_{0}$ so that 
\[
\|v_{\varepsilon}/d^{s}\|_{C^{\tilde{\alpha}}(B_{r/2}(x_{0})\cap\overline{U})}\le\tilde{C}(\|u\|_{L^{\infty}(\mathbb{R}^{n})}+\|f\|_{L^{\infty}(U)}+C_{0}),
\]
where $\tilde{C}$ depends only on $n,\lambda,\Lambda,s,r,U$. (Note
that we have rescaled the functions in order to obtain the estimate
on the domain $B_{r/2}(x_{0})$ instead of $B_{1/2}(0)$.) Now remember
that $u$ is the pointwise limit of a sequence $u_{\varepsilon_{l}}$
on $U$ (the diagonal sequence constructed in the proof of Theorem
\ref{thm: Reg u}). Considering this sequence in the above estimate,
it follows that there is a subsequence corresponding to $\varepsilon_{j}\to0$
such that $v_{\varepsilon_{j}}/d^{s}$ uniformly converges to a function
in $C^{\tilde{\alpha}}(B_{r/2}(x_{0})\cap\overline{U})$. But the
limit must be equal to $(u-\varphi)/d^{s}$. So $(u-\varphi)/d^{s}$
is $C^{\tilde{\alpha}}$ up to $\partial U$, and satisfies the desired
estimate.
\end{proof}
\begin{acknowledgement*}
This research was in part supported by Iran National Science Foundation
Grant No 97012372.
\end{acknowledgement*}
\bibliographystyle{plainnat}
\bibliography{/Volumes/A/Dropbox/Bibliography-Jan-2021}

\end{document}